\documentclass[preprint,12pt]{elsarticle}
\usepackage{amssymb}
\usepackage{amsmath}
\usepackage{amsthm}
\usepackage{bm}
\usepackage{algorithm}
\usepackage{algorithmic}
\usepackage{cases}
\usepackage{graphicx}
\usepackage{subcaption}
\usepackage{color}
\usepackage[colorlinks, citecolor=black]{hyperref}

\newcommand{\xkh}[1]{\left(#1\right)}
\newcommand{\dkh}[1]{\left\{#1\right\}}

\newcommand{\nj}[1]{\left \langle {#1} \right \rangle}

\newcommand{\norms}[1]{\left\|{#1}\right\|}
\newcommand{\abs}[1]{\lvert#1\rvert}
            
 \newcommand{\T}{\top}           
\newcommand{\dist}{{\rm dist}}

\newcommand{\va}{{\bm a}}
\newcommand{\vx}{{\bm x}}
\newcommand{\vy}{{\bm y}}

\newcommand{\vz}{{\bm z}}
\newcommand{\ve}{{\bm e}}
\newcommand{\R}{{\mathbb R}}

\newtheorem{theorem}{Theorem}[section]

\newtheorem{lemma}{Lemma}[section]
\newtheorem{remark}{Remark}[section]

\newtheorem{fact}{Fact}
\begin{document}
\begin{frontmatter}
\title{Convergence Analysis of Reshaped Wirtinger Flow with Random Initialization for Phase Retrieval}

\author[th]{Linbin Li}
\ead{linbinli@buaa.edu.cn}

\author[th]{Haiyang Peng}
\ead{haiyangpeng@buaa.edu.cn}

\author[th]{Yong Xia}
\ead{yxia@buaa.edu.cn}

\author[th]{Meng Huang}
\ead{menghuang@buaa.edu.cn}

\affiliation[th]{organization={School of Mathematics and Systems Science, Beihang University},
            addressline={Xueyuan Road 37},
            city={Beijing},
            postcode={100191},
            country={China}}

\begin{abstract}
This paper investigates phase retrieval using the Reshaped Wirtinger Flow (RWF) algorithm, focusing on recovering target vector $\vx \in \R^n$ from magnitude measurements 
\(
y_i = \left| \langle \va_i, \vx \rangle \right|, \; i = 1, \ldots, m,
\)
under random initialization, where $\va_i \in \R^n$ are measurement vectors.
For Gaussian measurement designs, we prove that when $m\ge O(n \log^2 n\log^3 m)$,  the RWF algorithm with random initialization achieves $\epsilon$-accuracy within
\(
O\big(\log n + \log(1/\epsilon)\big)
\)
iterations, thereby attaining nearly optimal sample and computational complexities comparable to those previously established for spectrally initialized methods. Numerical experiments demonstrate that the convergence rate is robust to initialization randomness and remains stable even with larger step sizes.

\end{abstract}
\end{frontmatter}
\section{Introduction}
\label{sec:introduction}
\subsection{Background and motivation}
The phase retrieval problem aims to recover a signal $\vx \in \mathbb{R}^n$ from magnitude-only measurements:
\begin{equation}
y_i = \left| \langle \va_i, \vx \rangle \right|, \quad i = 1, \ldots, m,
\label{problem}
\end{equation}
where $\vy =(y_1,\ldots,y_m)^\T$ denotes the observed measurement vector, and $\va \in \mathbb{R}^n$, for $i = 1, \ldots, m$, are known measurement vectors. The phase retrieval problem has long been recognized as a foundational challenge with deep historical roots in engineering, particularly in optical imaging and diffraction analysis. The phase recovery problem was first proposed by Sayre \cite{sayre1952some} and colleagues in 1952, playing a pivotal role in the discovery of the DNA double-helix structure. Shortly thereafter, in 1953, mathematician H. Hauptman \cite{karle1953solution} and collaborators developed an algorithm to address phase retrieval in crystallography, successfully applying it to elucidate molecular structures in crystalline materials. In 1978, Fienup \cite{fienup1978reconstruction} introduced an algorithm for phase retrieval from the Fourier modulus of two-dimensional images, incorporating constraints such as non-negativity and known image support. By 1984, it was well documented \cite{rosenblatt1984phase} that phase retrieval arises naturally in experiments involving diffracted electromagnetic radiation to determine the internal structure of objects. Around the turn of the millennium, renewed interest in phase retrieval emerged alongside advances in optical imaging techniques \cite{miao1999extending}. Today, phase retrieval enjoys broad applications across diverse fields, including optical imaging \cite{millane1990phase}, image processing \cite{holloway2016toward}, microscopy \cite{candes2015phase}, astronomical observation \cite{fienup1987phase}, X-ray crystallography \cite{miao1999extending}, microwave communications \cite{chen2020dual}, and quantum mechanics \cite{corbett2006pauli}.

A variety of algorithms have been developed to tackle the phase retrieval problem. Effective convex optimization methods for phase retrieval only emerged in the 21st century. Prior to that, the dominant approach was the simple yet efficient alternating projection method. The earliest alternating projection algorithm, known as the Gerchberg-Saxton (GS) algorithm, was proposed by Gerchberg and Saxton in 1972 \cite{gerchberg1994practical}. Building on this, Fienup and colleagues introduced two notable variants: the Basic Input-Output (BIO) algorithm \cite{fienup1978reconstruction} and the Hybrid Input-Output (HIO) algorithm \cite{fienup1982phase}. Subsequently, several new projection-based algorithms have been proposed, among which the Hybrid Projection-Reflection (HPR) algorithm \cite{bauschke2003hybrid} and the Relaxed Averaged Alternating Reflections (RAAR) algorithm \cite{luke2004relaxed}, both developed by Luke and collaborators, stand out. On the convex optimization front, prominent algorithms include PhaseLift \cite{candes2013phaselift}, PhaseCut \cite{waldspurger2015phase}, and others.

Recently, Cand$\Grave{\text{e}}$s et al.\ \cite{chen2015solving} introduces the Wirtinger Flow (WF) algorithm and proves signal recovery via the gradient algorithm with only ${O}(n\log n)$ Gaussian measurements, attaining $\epsilon$-accuracy within ${O}(mn^2\log\frac{1}{\epsilon})$ flops. The idea of the WF algorithm is to solve the following nonconvex least squares estimation problem (\ref{WF}) by vanilla gradient descent:
\begin{equation}
\min_{\vz\in\mathbb{R}^n} f(\vz)=\frac{1}{4m}\sum_{i=1}^m \left[(\va_i^{\T}\vz)^2-y_i^2\right]^2.
\label{WF}
\end{equation}
The Wirtinger Flow (WF) algorithm is further refined by the introduction of the Truncated Wirtinger Flow (TWF) algorithm, as proposed in \cite{chen2015solving}. Under Gaussian designs, the theoretical justification for the WF algorithm with random initialization is provided in \cite{chen2019gradient}. More recently, Zhang et al.\ \cite{zhang2016reshaped} introduced a nonsmooth yet computationally tractable loss function, developing the Reshaped Wirtinger Flow (RWF) algorithm with gradient-based updates. The RWF algorithm has been proven to achieve geometric convergence to a global optimum under carefully designed initialization for random Gaussian measurements, provided that the number of measurements $m$ is on the order of ${O}(n)$. The core idea of RWF is to solve the following problem (\ref{reshaped-WF}) by adopting the loss function $f(\vz)$:
\begin{equation}
    \min_{\vz \in \mathbb{R}^n} f(\vz) = \frac{1}{2m} \sum_{i=1}^m \left( \left| \va_i^{\T} \vz \right| - y_i \right)^2.
    \label{reshaped-WF}
\end{equation}

This paper considers the nonconvex and nonsmooth optimization problem (\ref{reshaped-WF}) and develops a gradient-like algorithm defined by
\begin{equation}
    \vz_{k+1} = \vz_k - \mu \nabla f(\vz_k).
    \label{gradient update}
\end{equation}
Here, $\nabla f(\vz)$ corresponds to the gradient of $f(\vz)$ when $\va_i^{\T} \vz_k \neq 0$ for all $i = 1, \ldots, m$. For samples at nonsmooth points where $\va_i^{\T} \vz_k = 0$, we utilize the Fréchet superdifferential \cite{kruger2003frechet} for nonconvex functions and set the corresponding gradient component to zero, since zero is an element of the Fréchet superdifferential. For simplicity, and with a slight abuse of terminology, we continue to refer to $\nabla f(\vz)$ as the ‘gradient’, which accurately characterizes the update direction in the gradient descent loop. Let $\sigma(a)$ denote the sign function, with the special case $\sigma(0) = 0$. The update direction is defined as follows:
\begin{equation}
    \nabla f(\vz) = \frac{1}{m} \sum_{i=1}^m \left( \left| \va_i^{\T} \vz \right| - y_i \right) \sigma(\va_i^{\T} \vz) \va_i.
    \label{define of gradient}
\end{equation}

The remarkable effectiveness of spectral initialization raises an intriguing question: is a carefully crafted initialization essential for achieving rapid convergence? It is clear that vanilla gradient descent cannot start from arbitrary points, as this may lead to entrapment in undesirable stationary points such as saddle points. Nevertheless, could there exist a simpler initialization strategy that avoids these stationary points while achieving comparable performance to spectral initialization? Motivated by this natural question, a strategy frequently favored by practitioners is to start from random initialization. The advantages of this approach are evident: unlike spectral methods, random initialization is model-agnostic and generally more robust to model mismatches. This important issue has attracted widespread attention.

However, the theoretical understanding of the RWF algorithm with random initialization remains limited. The objective of this paper is to analyze the RWF algorithm under random initialization using randomly resampled Gaussian measurements, providing rigorous theoretical support for this important and widely recognized problem. A key challenge arises from the presence of the sign function, which causes $\nabla f(\vz)$ to be discontinuous and not Lipschitz continuous near the origin.

\subsection{Our contributions}
In this paper, we prove that when measurement vectors $\{\va_i\}$ are i.i.d Gaussian random vectors, then RWF with random initialization converges to the target signal $\vx$ in ${O}(\log n + \log \frac{1}{\epsilon})$ iterations to $\epsilon$-accuracy. We fill the theoretical gaps of the RWF algorithm under random initialization, making the algorithm more practical and effective, and better meeting the requirements of real world problems. Unlike prior significant works \cite{zhang2016reshaped} where RWF relies on spectral initialization (computational complexity ${O}(n^3)$), this paper proves that under random initialization, the RWF algorithm can achieve the same effect as spectral initialization with ${O}(\log n)$ iterations, and its computational complexity is ${O}(n^2 \log n)$. Experiments also demonstrate the convenience and efficiency of random initialization.

In our analysis, we rigorously prove that the RWF convergence process admits a phase separation into two distinct regimes. In the first phase, we reveal that starting from a random initialization that is near orthogonal to the target signal, the angle between the RWF iteration and the true solution is monotonically decreasing. Moreover, our findings characterize the sophisticated variations of the signal component $\langle \vz_k,\vx\rangle \vx$ and the orthogonal component $\vz_k-\langle \vz_k,\vx\rangle \vx$ throughout this phase. We prove that this phase persists for ${O}(\log n)$ iterations with the sample complexity of order ${O}(n \log^2 n \log^{3} m)$. These theoretical derivations are subsequently verified through experiments. The second phase initiates when the distant between iteration and ground truth satisfies $ \dist(\vz,\vx)\leq \gamma$, where $\gamma
$ is a constant that can be set to 0.1 in experiments. During this phase, the iterates exhibit linear convergence to the ground truth, which can be deduced from the previous work \cite{zhang2016reshaped}.

The key to proving convergence of randomly initialized RWF lies in demonstrating that after the Phase 1, the iterative sequence enters a sufficiently small neighborhood of the true signal. The core difficulty stems from the discontinuity and the lack of Lipschitz continuity near the origin induced by the sign function and the absolute value function. During Phase 1, we implement a resampling strategy for the measurement vector $\{\va_i\}$ that enforces conditional independence between the iterate sequence $\vz_k$ and $\{\va_i\}$. This decoupling enables us to establish with high probability that the randomly initialized RWF algorithm avoids saddle-point trapping throughout Phase 1.

\subsection{Related work}
In the pursuit of developing nonconvex algorithms with guaranteed global performance for the phase retrieval problem, Netrapalli et al.\ \cite{netrapalli2013phase} proposed an alternating minimization algorithm, while Candès et al.\ \cite{candes2015phase}, Chen and Candès \cite{chen2015solving}, Zhang et al.\ \cite{zhang2016provable}, and Cai et al.\ \cite{cai2016optimal} focused on first-order gradient-based methods. A recent study by Sun et al.\ \cite{sun2018geometric} examined the geometric structure of the nonconvex objective and introduced a second-order trust-region algorithm. Additionally, Wei \cite{wei2015solving} empirically demonstrated the rapid convergence of a Kaczmarz stochastic algorithm.

This paper is most closely related to \cite{zhang2016reshaped} and \cite{chen2019gradient}. Unlike \cite{zhang2016reshaped}, we extend the RWF algorithm to random initialization, enhancing its applicability in practical engineering contexts. While we draw on the analysis of Lemma 8 in \cite{chen2019gradient}, our approach differs by adopting a distinct strategy to overcome the significant challenges introduced by the absolute value and sign functions. Nevertheless, we provide a more intuitive and accessible framework for understanding and analyzing the problem.

\subsection{Notations and outline}
We use boldface lowercase letters such as $\va$ to denote vectors, and lowercase letters such as $a$ to denote scalars. The notation $[\va]_i$ denotes the $i$-th component of the vector $\va$. For $a \in \mathbb{R}$, $\sigma(a)$ denotes the sign of $a$, with the special case $\sigma(0) = 0$. The identity matrix is denoted by $\mathbf{I}_n$, and $\mathcal{S}^{n-1}$ denotes the Euclidean sphere in $\mathbb{R}^n$. The vector $\ve_i$ denotes the standard basis vector whose $i^{\text{th}}$ component is 1 and all other components are 0. We adopt the convention that $\|\cdot\Vert$ denotes the $\|\cdot\Vert_2$ unless otherwise specified. The Euclidean distance between two vectors, accounting for a possible global sign difference, is defined as $ \dist(\vz,\vx)=\min\left\{\|\vz-\vx\Vert,\|\vz+\vx\Vert\right\}$. We use the notation $\gtrsim$ (respectively, $\lesssim$) to indicate the following: if $x \gtrsim y$ (resp. $x \lesssim y$), then there exists a constant $c > 0$ such that $x \geq c y$ (resp. $x \leq c y$).

\subsection{Organization}
The paper is organized as follows. Section 2 introduces the proposed algorithm and presents the main theoretical contributions. Section 3 provides the proofs of the main results. Section 4 presents numerical experiments that validate our theoretical findings, accompanied by detailed discussions. Section 5 presents the conclusions and discussion. Section 6 covers all necessary preliminaries and foundational concepts.  Finally, Sections 7 and 8 contain the full technical proofs supporting all major claims made in this paper.
\section{Algorithm and main results}
\label{sec:Algorithm and main results}

\begin{algorithm}[!ht]
\renewcommand{\algorithmicrequire}{ \textbf{Input:}}
\renewcommand{\algorithmicensure}{\textbf{Output:}}
\caption{Resampled Reshaped Wirtinger Flow with random initialization}
\label{algorithm}
\begin{algorithmic}[0]
\REQUIRE Data $\dkh{\va_i,y_i}_{i=1}^{\tilde{m}K},i=1,\ldots, m$ equally partitioned into $K$ disjoint blocks $\dkh{\va_i,y_i}_{i \in \mathcal I_k}, k=1,\ldots, K$; random initialization $\vz_0 \sim \mathcal N({\bm 0}, n^{-1}{\bm I}_n)$; constant step size $\mu\in(0,0.5]$ (set to 0.5 in experiments); the maximum number of iterations $T$.
\FORALL{$k=0,1,\ldots,T$}
\STATE 
\[
 \vz_{k+1}=\vz_{k}-\frac{\mu}{\tilde{m}}\sum_{i \in \mathcal I_{t}}(| \va_i^{\T}\vz_k\vert-y_i)\sigma(\va_i^{\T}\vz_k)\va_i,
 \]
 where $t=\mbox{mod}(k,K)$.
\ENDFOR 
\ENSURE $\vz_T$.
\end{algorithmic}
\end{algorithm}

In this work, we consider the algorithm in a batch setting. We partition the sampling vectors $\va_i$ and their corresponding observations $y_i$ into $K$ disjoint blocks $\dkh{\va_i,y_i}_{i \in \mathcal I_k}$, for $k = 1, \ldots, K$, each of roughly equal size, and perform gradient descent cyclically over these blocks. We set the sample size per block as $\tilde{m}$, namely $m=\tilde{m}K$. The algorithm is summarized in Algorithm \ref{algorithm}. Unlike the RWF algorithm in \cite{zhang2016reshaped}, which is using spectral initialization, our approach employs random initialization followed by resampling during the first $T_{\gamma}$ steps. After these $T_{\gamma}$ steps, the iterate $\vz_{T_{\gamma}}$ satisfies the initialization conditions required by the RWF algorithm, enabling linear convergence in subsequent iterations, analogous to the original RWF method. It is worth noting that the total number of samples $m$ satisfies $m\geq {O}({n\log^2 n\log^3 m})$.

\begin{theorem}
Fixed $\vx\in\mathbb{R}^{n}$. Suppose  that $K=C_0 \log n$ for some constant $C_0>0$ and  $\va_i \sim \mathcal N({\bm 0}, {\bm I}_n), i=1,\ldots,\tilde{m}K$ are i.i.d. Gaussian random vectors and $\mu_t\equiv \mu=c/ \|\vx\Vert$ for some  sufficiently small constant $c>0$. Assume that the random initialization $\vz_0$ is independent of $\{\va_i\}$ and obeys
    \begin{equation} \label{initial}
        \frac{|\langle\vz_0,\vx\rangle\vert}{\|\vx\Vert} \geq \frac{1}{2\sqrt{n\log n}} \quad and \quad (1-\frac{1}{\log n})\|\vx\Vert\leq\|\vz_0\Vert\leq (1+\frac{1}{\log n})\|\vx\Vert.
    \end{equation}
    Then  there exist a sufficiently small absolute constant $0<\gamma<1$ and $T_\gamma \lesssim \log n$ such that with probability at least $1-O(\exp\xkh{-c_1n})-O(\tilde{m}^{-10})$, the RWF algorithm obeys
    \begin{equation*}
        \dist(\vz_k,\vx)\leq \gamma(1-\rho)^{k-T_{\gamma}}\|\vx\Vert, \quad \forall k \geq T_{\gamma}
    \end{equation*}
    for some universal constant $0<\rho<1$, provided $\tilde{m}\geq Cn\log{n}\log^3{\tilde{m}}$ for some sufficiently large constant $C>0$.  Here,  $c_1$ is an absolute constant.
\label{main theorem} 
\end{theorem}

\begin{remark}
As shown in \cite{chen2019gradient}, random initialization $\vz_0 \sim \mathcal N({\bm 0}, n^{-1} \norms{\vx}^2 {\bm I}_n)$ obeys the condition \eqref{initial} with probability at least $1-O(1/\sqrt{\log n})$.
\end{remark}

\section{Proof of main result}

 Due to the rotational invariance of Gaussian distributions, throughout the analysis, we assume without loss of generality that 
\begin{equation*}
    \vx=\ve_1 \quad \mbox{and} \quad \vz_0^{||}=[\vz_0]_1>0.
    \label{without loss of generality}
\end{equation*}
For each $\vz_k$, we decompose 
\begin{equation*}
\vz_k= \vz_k^{||} \ve_1 +\vz_k^{\perp},
    \label{def:zk_||}
\end{equation*}
where $\vz_k^{||}=\nj{\vz_k,\ve_1}$. Furthermore, we define
\begin{equation*}
    \alpha_k=\vz_k^{||},  \quad  \beta_k=\|\vz_k^{\perp}\Vert \quad \mbox{and} \quad  r_k = \norms{\vz_k}. 
    \label{def:alpha beta}
\end{equation*}
To prove that $\vz_k$ converges to $\vx$, we just need show that $\alpha_k$ converges to $1$ and $\beta_k$ converges to $0$.  
To this end, observing that if  there exists some $T_{\gamma} \in \mathbb N_{+}$ such that 
\begin{equation}   \label{eq:localrg}
    |\alpha_{T_{\gamma}}-1\vert\leq\frac{1}{2}\gamma \quad \mbox{and} \quad \beta_{T_{\gamma}}\leq\frac{1}{2}\gamma
\end{equation}
for some sufficiently small constant $\gamma>0$, 
then 
\begin{equation*} 
\dist(\vz_{T_{\gamma}},\vx)\leq \|\vz_{T_{\gamma}}-\vx\Vert\leq|\alpha_k-1\vert+\beta_k\leq\gamma.
\end{equation*}
It means $\vz_{T_{\gamma}}$ is very close to the target solution $\vx$. Therefore, after $k \ge T_\gamma$,  we can invoke prior theory \cite{zhang2016reshaped} to show that there exists constant $0<\rho<1$ such that with probability at least $1-{O}(\exp(-c_1m))$,
\begin{equation} \label{eq:phase2}
    \dist(\vz_k,\vx)\leq (1-\rho)^{k-T_{\gamma}} \norms{\vz_{T_\gamma}-\vx},\quad \forall k\geq T_{\gamma}.
\end{equation}
To complete the proof, we only need to show that  \eqref{eq:localrg} holds for some $T_{\gamma}={O}(\log n)$. 
Our proof will be inductive in nature.  Specifically, we will first identify a induction hypothesis that 
\begin{equation}\label{induction hypothesis}
    c_{l}\leq\|\vz_k^{\perp}\Vert\leq \|\vz_k\Vert\leq 2
\end{equation}
for a universal constant $c_l >0$, and then we prove the approximate state evolution:
\begin{subequations}\label{alpha beta approximate}
\begin{align}
	\label{alpha beta approximate:a} \alpha_{k+1} &= \left[1- \mu\left(1- \frac{2}{\pi} \cdot \frac{\beta_k}{\alpha_k^2+\beta_k^2}+\zeta_k\right)\right]\alpha_k+\frac{2\mu}{\pi} \arcsin{\frac{\alpha_k}{\sqrt{\alpha_k^2+\beta_k^2}}}, \\
	\label{alpha beta approximate:b} \beta_{k+1} &= \left[1- \mu\left(1-\frac{2}{\pi} \cdot \frac{\beta_k}{\alpha_k^2+\beta_k^2}+\rho_k\right)\right]\beta_k,
\end{align}
\end{subequations}
where $|\zeta_k\vert\leq \frac{c}{\log \tilde{m}}$ and $|\rho_k\vert\leq \frac{c}{\log \tilde{m}}$ for some universal constant $c>0$. Based on state evolution \eqref{alpha beta approximate},  the inequality  \eqref{eq:localrg} will hold for some $T_{\gamma}={O}(\log n)$, which gives the conclusion of Theorem \ref{main theorem}. Finally, we proceed by establishing the  hypothesis \eqref{induction hypothesis} via induction.

The next lemma shows that if \eqref{induction hypothesis}  hold for the $k$-th iteration, then $\alpha_{k+1}$ and $\beta_{k+1}$ follow the approximate  state evolution \eqref{alpha beta approximate}.
\begin{lemma}\label{lemma:bound}
\setlength{\belowdisplayskip}{1pt}
    Suppose $\tilde{m}\geq Cn \log n\log^3{\tilde{m}}$ for some sufficiently large constant $C>0$. Assume that  $\va_i\overset{i.i.d}{\thicksim}\mathcal{N}({\bm 0},{\bm I}_{n}),i=1,\ldots,\tilde{m}$ are Gaussian random vectors. For any $0 \le k\le T$ with $T\lesssim\log n$, if the $k$th iterates  satisfy the induction hypotheses (\ref{induction hypothesis}), then with probability at least $1-{O}(\exp({-\hat{c}n}))-{O}(\tilde{m}^{-20})$, 
\begin{align*}
	\alpha_{k+1} &= \left[1- \mu\left(1- \frac{2}{\pi} \cdot \frac{\beta_k}{\alpha_k^2+\beta_k^2}+\zeta_k\right)\right]\alpha_k+\frac{2\mu}{\pi} \arcsin{\frac{\alpha_k}{\sqrt{\alpha_k^2+\beta_k^2}}}, \\
	\beta_{k+1} &= \left[1- \mu\left(1-\frac{2}{\pi} \cdot \frac{\beta_k}{\alpha_k^2+\beta_k^2}+\rho_k\right)\right]\beta_k,
\end{align*}
hold for some  $|\zeta_k\vert\leq \frac{c}{\log \tilde{m}}$ and $|\rho_k\vert\leq \frac{c}{\log \tilde{m}}$,  where $c>0$ is a sufficiently small constant and $\hat{c}>0$ is a universal constant.  Furthermore, denote $\omega_k:=\arctan\xkh{\alpha_k/\beta_k}$. Then 
\begin{equation}\label{omega} 
    \tan{\omega_{k+1}} \ge (1+\frac{1}{4}\mu)\tan{\omega_k}. 
\end{equation}
\end{lemma}
\begin{proof}
\setlength{\topsep}{0pt}
See Section \ref{sec:Proofs}.
\end{proof}

We now proceed to demonstrate that the induction hypothesis  (\ref{induction hypothesis}) holds for all $0\leq k \leq T$. 
The base case is readily verified due to the identical initial conditions provided in (\ref{initial}). 
\begin{lemma}\label{bound of r_k}
\setlength{\belowdisplayskip}{0pt}
    Suppose $\tilde{m}\geq Cn \log n\log^3{\tilde{m}}$ for some sufficiently large constant $C>0$. If the induction hypothesis (\ref{induction hypothesis}) hold true up to the $k$-th iteration for some $k\leq T$, then one has
    \begin{equation*}
    c_{l}\leq\|\vz_{k+1}^{\perp}\Vert \leq \|\vz_{k+1}\Vert \leq 2.
    \end{equation*}
\end{lemma}
\begin{proof}
\setlength{\topsep}{0pt}
    See Section \ref{proof for bound of r_k}.
\end{proof}

Now, we can give the proof of Theorem \ref{main theorem}.

\begin{proof}[Proof of Theorem \ref{main theorem}]
Let \begin{equation}\label{def:t gamma}
        T_{\gamma}=\min\left\{k:\left|1-\alpha_k \right\vert\leq\frac{1}{2}\gamma \enspace \mbox{and}  \enspace \beta_k \leq\frac{1}{2}\gamma \right\}.
    \end{equation}
According to \eqref{eq:phase2}, we only need to show $T_{\gamma}\lesssim\log n$. To this end, define
\begin{equation*}
    T_{\gamma,2}=\min\left\{k:\beta_{k+1}\leq\frac{1}{2}\gamma\right\}.
\end{equation*}
Under the initialization conditions (\ref{initial}), it follows from Lemma \ref{lemma:bound} and Lemma \ref{bound of r_k} that  the properties (\ref{alpha beta approximate}) and \eqref{omega} hold for all $0\le k\le T$. Observe that 
    \begin{equation*}
        \tan{\omega_0}=\frac{\alpha_0}{\beta_0}={O}(\frac{1}{\sqrt{n\log n}}).
    \end{equation*}
    Using (\ref{omega}), one has
    \begin{align*}
        \tan{\omega_{k}}\geq (1+\frac{1}{4}\mu)^{k}\tan{\omega_0}.
    \end{align*}
Define a critical moment regarding $\omega_k$ as 
\begin{equation}
    T_{\omega}=\min\left\{k:\omega_{k+1}\geq\frac{\pi}{2}-\frac{1}{4}\gamma\right\}.
\end{equation}
    Let $k_0$ satisfies $(1+\frac{1}{4}\mu)^{k_0}\tan{\omega_0}= \tan{(\frac{\pi}{2}-\frac{1}{4}\gamma)}$. One can check that  
    \begin{align*}
        k_0={O}\xkh{\frac{\log{\frac{1}{\gamma}}\sqrt{n\log n}}{\log{(1+\frac{1}{4}\mu)}}}={O}(\frac{\log n}{\mu}).
    \end{align*}
    This further implies that
    \begin{equation}
        \omega_k\geq\frac{\pi}{2}-\frac{1}{4}\gamma, \quad \forall k\geq k_0 \quad \mbox{and} \quad T_{\omega}+1\leq k_0 \lesssim\log n.
    \end{equation}
    Moreover, when $k \geq T_{\omega}+1$ , since $\cos{\omega_{k}}\leq\cos{(\frac{\pi}{2}-\frac{1}{4}\gamma)}\leq\frac{1}{4}\gamma$, combined with Lemma \ref{bound of r_k}, one has 
    \begin{equation}
        \beta_k=r_k\cos{\omega_k}\leq \frac{1}{2}\gamma \quad \mbox{and} \quad
        T_{\gamma,2}+1\leq T_{\omega}+1\lesssim \log n. 
    \end{equation}

    Next, we show that $T_{\gamma}\lesssim\log n$.  Actually, if $T_{\gamma}\leq T_{\omega}$ then it holds trivially. Otherwise, we prove 
    \begin{equation} \label{eq:TminuTgam}
    T_{\gamma}-T_{\omega}\lesssim{O}(\frac{1}{\mu}).
    \end{equation}
 For any $T_{\omega}+1\leq k\leq T_{\gamma}$, we have established that $\omega_k\geq \frac{\pi}{2}-\frac{1}{4}\gamma$. Consequently, the triangle inequality further derives that
    \begin{align*}
        \|\vz_{k+1}-\vx\Vert &= \norms{\vz_k-\vx-\mu\left[\vz_k-(1-\frac{2\theta_k}{\pi})\vx\right]+\mu\frac{2\sin{\theta_k}}{\pi}\frac{\vz_k}{\|\vz_k\Vert}+\mu r(\vz_k)} \\
        &\leq (1-\mu)\norms{\vz_k-\vx} + \frac{2}{\pi}\mu \theta_k\|\vx\Vert + \frac{2}{\pi}\mu \theta_k |\sin{\theta_k}\vert+\mu \|r(\vz_k)\Vert \\
        &\leq (1-\mu)\norms{\vz_k-\vx} +\mu \left(\frac{1}{\pi}\gamma+\frac{c}{\log \tilde{m}}\right).
    \end{align*}
By recursion, we establish
    \begin{equation*}
        \|\vz_{T_{\gamma}}-\vx\Vert \leq(1-\mu)^{T_{\gamma}-T_{\omega}-1}\|\vz_{T_{\omega}+1}-\vx\Vert + \left(\frac{4}{\pi}\epsilon+\frac{c}{\log \tilde{m}}\right).
    \end{equation*}
    Note that $\frac{1}{\pi}\gamma+\frac{c}{\log \tilde{m}}\leq \frac{1}{2}\gamma$. Lemma \ref{bound of r_k} gives $\|\vz_{T_{\omega}+1}-\vx\Vert \leq \|\vz_{T_{\omega}+1}\Vert+\|\vx\Vert\leq 4$. Moreover, it holds  
\[
(1-\mu)^{T_{\gamma}-T_{\omega}-1}\|\vz_{T_{\omega}+1}-\vx\Vert \leq \frac{1}{2}\gamma,
\]
which gives $T_{\gamma}-T_{\omega}-1\lesssim \frac{1}{\mu}$. Then we complete the proof that $T_{\gamma}=T_{\omega}+1+(T_{\gamma}-T_{\omega}-1)\lesssim \log n$.

\end{proof}

\section{Numerical performance and Conclusion}
\subsection{Numerical performance}
\noindent In this section, our numerical experiments demonstrate that the randomly initialized RWF algorithm can converge to a neighborhood of the target vector $\vx$ through the implementation of ${O}(\log n)$ iterations. Furthermore, we conduct a comprehensive performance evaluation of the Reshaped Wirtinger Flow (RWF) algorithm against the conventional Wirtinger Flow (WF) method under random initialization conditions. All numerical experiments are conducted using Matlab 2020a and carried out on a computer equipped with Intel Core i7 2.30GHz CPU and 32GB RAM. 
\begin{figure}[H]
    \centering
    \begin{subfigure}[b]{0.48\textwidth} 
        \centering
        \includegraphics[width=\linewidth]{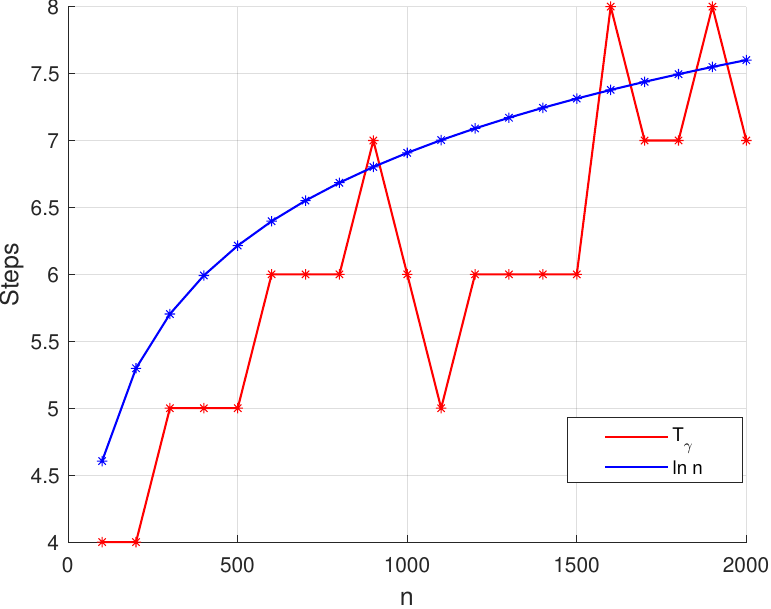}
        \caption{$\gamma=0.5$}
        \label{fig:gamma=0.5}
    \end{subfigure}
    \hfill 
    \begin{subfigure}[b]{0.48\textwidth}
        \centering
        \includegraphics[width=\linewidth]{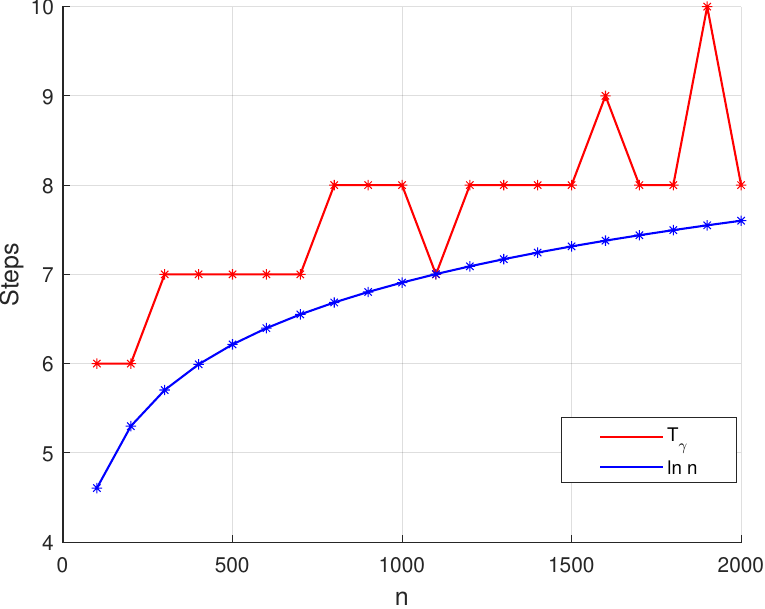}
        \caption{$\gamma=0.1$}
        \label{fig:gamma=0.1}
    \end{subfigure}
    \caption{This figure demonstrates the evolution of $T_{\gamma}$ under Algorithm \ref{algorithm} with the parameter setting $m=10n$.}
    \label{fig:gamma}
\end{figure}
\noindent Figure \ref{fig:gamma} presents the evolution of the stopping time $T_{\gamma}$ for $\gamma=\frac{1}{2}$ and $\gamma=\frac{1}{10}$ respectively with n ranging from 100 to 2000 (where $m = 10n$). The results demonstrate that $T_{\gamma}$ scales with ${O}(\log n)$, which is consistent with our theoretical conclusion. 

\begin{figure}[htbp]
    \centering
    \includegraphics[width=0.5\linewidth]{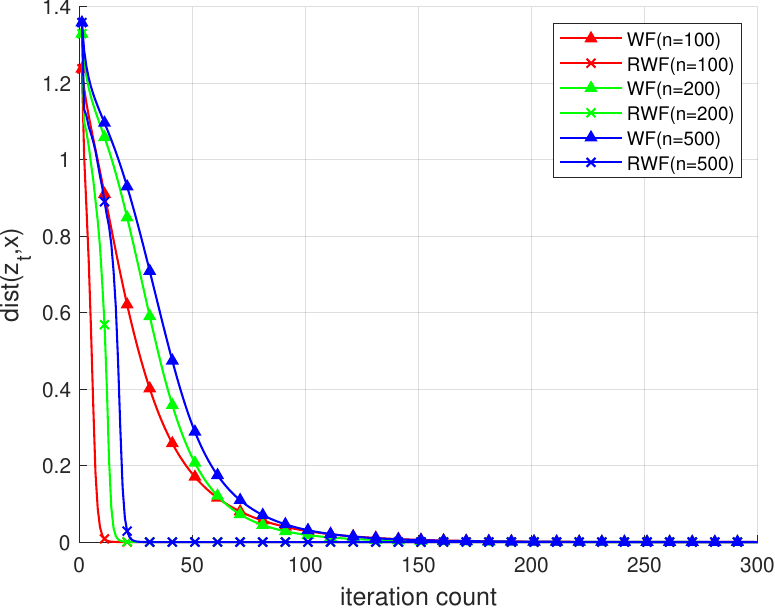}
    \caption{Comparison of sample complexity among WF and RWF.}
    \label{fig:enter-label}
\end{figure} 
\noindent Figure \ref{fig:enter-label} compares the convergence rates of the WF and RWF algorithms under random initialization conditions. For WF algorithm, we adopt the parameter settings from \cite{chen2019gradient} as step size $\mu_t\equiv 0.1$, while the RWF algorithm is implemented with $\mu_k=0.5$ and $\gamma=0.5$. All experiments are conducted under identical conditions, including the same random initial point $\vz_0$ and sampling ratio $m=10n$, with problem dimensions $n \in \{100, 200, 500\}$. The results demonstrate that the RWF algorithm achieves significantly faster convergence.

\begin{figure}[htbp]
    \centering
    \includegraphics[width=0.5\linewidth]{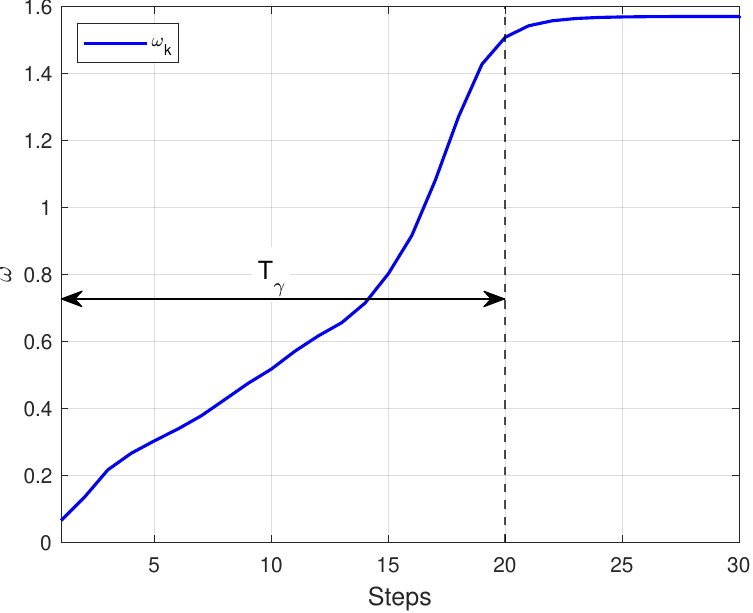}
    \caption{Dynamical evolution of $\omega_k$ during Phase 1 under parameter configuration $n=500$ and $m=10n$.}
    \label{fig:omega}
\end{figure}

\begin{figure}[htbp]
    \centering
    \includegraphics[width=0.5\linewidth]{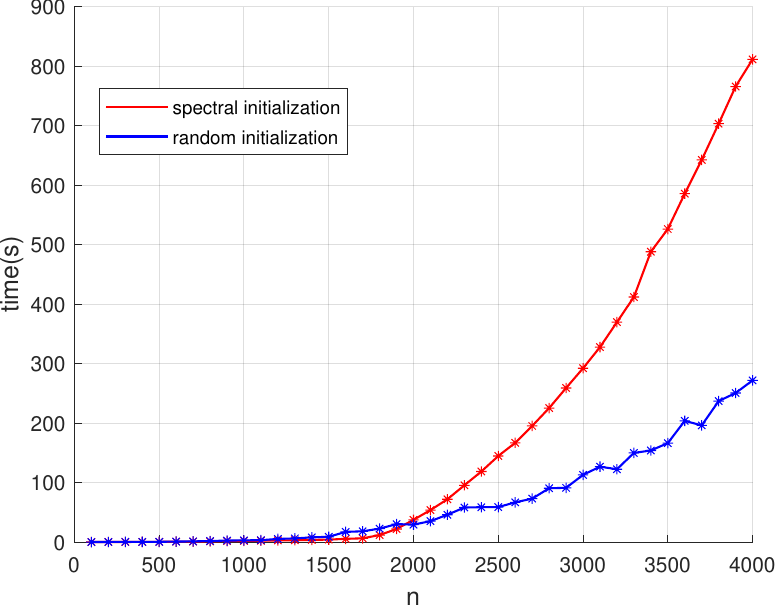}
    \caption{Comparison of the time required for Algorithm 1 and spectral initialization to satisfy $\dist(\vz,\vx)\leq \gamma$, where $\gamma=0.1$ and $m=10n$.}
    \label{fig:time for initialization}
\end{figure}

\noindent Figure \ref{fig:omega} illustrates the evolution of $\omega_k$, consistent with (\ref{omega}) in Lemma \ref{lemma:bound} under parameter configuration $\mu=0.5$ and $\gamma=0.1$. The results reveal that during Phase 1, the complementary angle between the iterative sequence $\vz_k$ and the target vector $\vx$ exhibits monotonic increase. This angular progression demonstrates geometrically that the iterates converge gradually toward $\vx$.

\noindent Figure \ref{fig:time for initialization} shows the comparison of the time required in Algorithm \ref{algorithm} to achieve the same accuracy as spectral initialization. The ordinate represents the time spent, in seconds. The step size $\mu$ in Algorithm 1 sets 0.5. It can be seen from the figure that as $n$ increases, the time required for spectral initialization increases significantly, which is because spectral initialization needs to compute the principal eigenvector of the matrix. The RWF algorithm using random initialization has a significant advantage in terms of time.

\section{Conclusion and discussion}
This paper investigates the convergence guarantees of the Reshaped Wirtinger Flow (RWF) algorithm under random initialization. Our theoretical analysis shows that, with random initialization, the RWF algorithm achieves $\epsilon$-accuracy within \(
O\big(\log n + \log(1/\epsilon)\big)
\) iterations, provided the sampling ratio $m$ is on the order of ${O}(n \log^2 n \log^3 m)$. Numerical experiments not only validate the predicted convergence rate but also demonstrate that after approximately ${O}(\log n)$ iterations, the algorithm enters a sufficiently small constant neighborhood around the target vector $\vx$.

The current analysis has three main limitations, which suggest natural directions for future work: (i) our theory assumes noiseless measurements, whereas practical applications often involve noise and thus require robust extensions; (ii) the assumption that measurement vectors $\va_i$ follow a Gaussian distribution could be generalized to sub-Gaussian distributions to broaden applicability; and (iii) the presence of sign and absolute value functions introduces significant technical challenges, including loss of continuity and separability. Although our resampling-based analysis effectively handles these nonsmooth components, it does so at the expense of moderately increased sample complexity. Developing more refined analytical techniques may reduce this requirement. These directions offer promising avenues to strengthen both the theoretical foundation and practical utility of the RWF algorithm.

\section{Preliminaries}
\begin{lemma} 
    Assume that $\va_i\sim \mathcal{N}(\mathbf{0},\mathbf{I}_n)$ is a Gaussian random vector. For any fixed $\vx$, $\vz\in \mathbb{R}^n$, assuming that $\theta\in (0,\frac{\pi}{2}) $ is the angle between $\vx$ and $\vz$, one has
    \begin{equation}
        \mathbb{E}_{\va_i\sim \mathcal{N}(\mathbf{0},\mathbf{I}_n)}\sigma(\va_i^\T\vz) |\va_i^\T\vx\vert \va_i =\left(1-\frac{2\theta}{\pi}\right)\vx+\frac{2\sin{\theta}}{\pi}\frac{\vz}{\|\vz\Vert}\|\vx\Vert.
    \end{equation}
\label{expectation}    
\end{lemma}
\vspace{-1.5\baselineskip}
\begin{proof}
Due to the homogeneous, we assume  that  $\vx,\vz\in\mathbb{S}^{n-1}$.  Without loss of generality, we assume $\vz=\ve_1$ and $\vx=(\cos{\theta},\sin{\theta},0,\ldots,0)^{\T}$.
Therefore, we obtain
\begin{align*}
    \mathbb{E}_{\va_i\sim \mathcal{N}(\mathbf{0},\mathbf{I}_n)}\sigma(\va_i^\T\vz) |\va_i^\T\vx\vert \va_i  = \mathbb{E}_{\va_i\sim \mathcal{N}(\mathbf{0},\mathbf{I}_n)} \sigma(a_{i,1})|a_{i,1}\cos{\theta}+a_{i,2}\sin{\theta}\vert\va_i.   
\end{align*}
Let $a_{i,1}=r\cos{\alpha}$ and $a_{i,2}=r\sin{\alpha}$ where $r\in (0,+\infty)$, $\alpha\in(0,2\pi)$. It is evident that in the above equation the $3$-rd to $n$-th components are zero. Hence, only the first and second components need to be computed. Firstly, we calculate the first component by
\begin{align*}
    &\int_{0}^{2\pi}\int_0^{+\infty}\frac{1}{2\pi} e^{-\frac{1}{2}r^2}r\sigma(r\cos{\alpha})|r\cos{(\alpha-\theta)}\vert r\cos{\alpha} d\alpha dr \\
    & = \int_{0}^{+\infty}\frac{1}{2\pi}r^3 e^{-\frac{1}{2}r^2} dr  \int_{0}^{2\pi}|\cos{(\alpha-\theta)}\cos\alpha\vert d\alpha =\frac{2\sin{\theta}}{\pi}+(1-\frac{2\theta}{\pi})\cos{\theta}.
\end{align*}
Secondly, a similar calculation is performed for the second component as following.
\begin{align*}
    \int_{0}^{2\pi}\int_0^{+\infty} \frac{1}{2\pi}e^{-\frac{1}{2}r^2}r\sigma(r\cos{\alpha})|r\cos{(\alpha-\theta)}\vert r\sin{\alpha}  dr d\alpha =(1-\frac{2\theta}{\pi})\sin{\theta}.
\end{align*}
This completes the proof.
\end{proof}

\begin{lemma}\cite[Corollary 2.8.3]{vershynin2020high}\label{bernstein}
Let $X_1,\ldots, X_m$ be independent, mean zero, sub-exponential random variables. Then, for every $t\geq 0$, one has
$$ \mathbb{P}\left\{ \left| \frac{1}{m}\sum_{i=1}^m X_i \right\vert \geq t \right \} \leq 2\exp{\left[-c\min\left(\frac{t^2}{K^2},\frac{t}{K}\right)m\right]}$$
where $K=\max_i{\| X_i \Vert_{\psi_1}}$.    
\label{theorem 2.8.3 Bernstein's inequality}
\end{lemma}
\begin{lemma}\cite[Corollary 4.2.13]{vershynin2020high}
    The covering numbers of the Euclidean ball $\mathcal{B}_2^n$ satisfy the following for any $\epsilon >0$, \begin{equation*}
        \left(\frac{1}{\epsilon}\right)^n \leq \mathcal{N}(\mathcal{B}_2^n,\epsilon)\leq \left(\frac{2}{\epsilon}+1\right)^n.
    \end{equation*}
    The same upper bound is true for the Euclidean sphere $\mathcal{S}^{n-1}$.
\label{theorem: eplison-net}
\end{lemma}
\begin{lemma}\cite[Exercise 4.4.2]{vershynin2020high}
    Let $\vx\in\mathbb{R}^n$ and $\mathcal{N}$ be an $\epsilon$-net of the sphere $\mathcal{S}^{n-1}$, one has
    \begin{equation}
        \sup_{\vy\in\mathcal{N}}\langle\vx,\vy\rangle\leq \|\vx\Vert_2\leq \frac{1}{1-\epsilon}\sup_{\vy\in\mathcal{N}}\langle\vx,\vy\rangle.
    \end{equation}
    \label{exercise 4.4.2}
\end{lemma}
   A user-friendly version of the Bernstein inequality is as follows.
\begin{lemma}\cite[Lemma 11]{chen2019gradient}
    Consider $m$ independent random variables $w_l(1\leq l\leq m)$, each satisfying $| w_l \vert \leq B$. For each $a\geq 2$, one has 
\begin{equation}
    \left| \sum_{l=1}^m w_l-\sum_{l=1}^m \mathbb{E}[w_l] \right\vert \leq \sqrt{2a\log m\sum_{l=1}^m\mathbb{E}[w_l^2]}+\frac{2a}{3}B\log m
\end{equation}
with probability at least $1-2m^{-a}$.
\label{Chen 2019, lamma 11}
\end{lemma}

The standard concentration inequality reveals that 
\begin{equation}\label{eq:max_a_1}
    \max_{1\leq i\leq m}|\va_i^{\T}\vx\vert =\max_{1\leq i\leq m} |a_{i,1}\vert \le 10\sqrt{\log m}
\end{equation}
with probability at least $1-{O}(m^{-20})$.
Additionally, apply the standard concentration inequality to see that 
\begin{equation}\label{eq:max_a_2}
    \max_{1\leq i\leq m} \| \va_i\Vert_2 \leq \sqrt{6n}
\end{equation}
with probability at least $1-{O}(m\exp{(-1.5n)})$.

\section{Proof of Lemma \ref{lemma:bound}}\label{sec:Proofs}
Consider the update rule 
\begin{equation*}
    \vz_{k+1}=\vz_{k}-\mu \nabla f(\vz_k),
\end{equation*}
where  
\[
\nabla f(\vz_k)=\frac{1}{\tilde{m}}\sum_{i=1}^{\tilde{m}} (| \va_i^{\T}\vz_k\vert-| \va_i^\T \vx\vert)\sigma(\va_i^{\T}\vz_k)\va_i,
\]
According to Lemma \ref{expectation}, one has
\begin{equation} \label{eq:expc}
    \nabla F(\vz): =\mathbb{E}_{\va_i\sim \mathcal{N}(\mathbf{0},\mathbf{I}_n)}\nabla f(\vz)=\vz-\left(1-\frac{2\theta}{\pi}\right)\vx-\frac{2\sin{\theta}}{\pi}\frac{\vz}{\|\vz\Vert}
\end{equation}
where $\theta\in [0,\frac{\pi}{2}]$ is the angle between $\vx$ and $\vz$.  Therefore, 
\begin{equation*}
    \vz_{k+1}=\vz_{k}-\mu \nabla F(\vz_k) + \mu \underbrace{\left(\nabla F(\vz_k)-\nabla f(\vz_k)\right)}_{:=r(\vz_k)}.  
    \label{iteration}
\end{equation*}
The definition (\ref{def:zk_||}) implies the following decompositions:
\begin{equation*}
    \vz_k^{||} =\langle\vz_k,\vx\rangle  \quad \mbox{and} \quad {\vz_k^{\perp}}= \vz_k-\langle\vz_k,\vx\rangle\vx.
\end{equation*}
Combine \eqref{eq:expc} and \eqref{iteration} to obtain
\begin{subequations}\label{zk+1}
\begin{align}
	\label{zk+1:1}\vz_{k+1}^{||} &=\vz_k^{||}-\mu\vz_k^{||}+\mu\left(1-\frac{2\theta_k}{\pi}\right)+\mu\frac{2\sin{\theta_k}}{\pi}\frac{\vz_k^{||}}{\|\vz_{k}\Vert}+\mu r_1(\vz_k) \\
	\label{zk+1:2}\vz_{k+1}^{\bot} &=\vz_k^{\bot}-\mu\vz_k^{\bot}+\mu\frac{2\sin{\theta_k}}{\pi}\frac{\vz_k^{\bot}}{\|\vz_{k}\Vert}+\mu r_{\bot}(\vz_k)
\end{align}
\end{subequations}
where
\begin{align*}
    r_1(\vz_k) & = \underbrace{\vz_k^{||}-\frac{1}{\tilde{m}}\sum_{i=1}^{\tilde{m}}\va_i^{\T}\vz_k\va_{i,1}}_{:=J_1} + \underbrace{\frac{1}{\tilde{m}}\sum_{i=1}^{\tilde{m}}|\va_i^\T\vx\vert\sigma(\va_i^\T\vz_k)\va_{i,1}-1+\frac{2\theta_k}{\pi}-\frac{2\sin{\theta_k}}{\pi}\frac{\vz_k^{||}}{\|\vz_k\Vert}}_{:=J_2} \\
    r_{\perp}(\vz_k) & = \underbrace{\vz_k^{\perp}-\frac{1}{\tilde{m}}\sum_{i=1}^{\tilde{m}}\va_i^{\T}\vz_k\va_{i,\perp}}_{:=J_3} + \underbrace{\frac{1}{\tilde{m}}\sum_{i=1}^{\tilde{m}}|\va_i^\T\vx\vert\sigma(\va_i^\T\vz_k)\va_{i,\perp}-\frac{2\sin{\theta_k}}{\pi}\frac{\vz_k^{\perp}}{\|\vz_k\Vert}}_{:=J_4}.  
\end{align*}
Here, $\va_{i,1} \in \R$ is the first entry of $\va_i$,  and $\va_{i,\perp}=(0, a_{i,2},\ldots, a_{i,n})^\T \in \R^n$, and $\sin{\theta_k}$ is defined as
\begin{equation}
    \sin{\theta}_k=\frac{\|\vz_k^{\perp} \Vert}{\|\vz_k\Vert}.
\end{equation}

    In the following, we analyze the upper bounds of the four terms $J_1$, $J_2$, $J_3$ and $J_4$ separately to control \eqref{zk+1:1} and \eqref{zk+1:2}.  
\begin{itemize}
    \item For the term $J_1$, decompose it as follows
    \begin{equation*}
    \begin{matrix}
        J_1 =\underbrace{\left(1-\frac{1}{\tilde{m}}\sum_{i=1}^{\tilde{m}}\va_{i,1}^2\right)\vz_k^{||}}_{:=J_{11}} - \underbrace{\frac{1}{\tilde{m}}\sum_{i=1}^{\tilde{m}}\va_{i,\perp}^{\T}\vz_k^{\perp}\va_{i,1}}_{:=J_{12}}.
        \end{matrix}
    \end{equation*}
    We first consider $J_{11}$. Lemma \ref{bernstein} gives that, for any $t \geq 0$, it holds
    \begin{equation*}
        \mathbb{P}\left\{ \left| \frac{1}{\tilde{m}}\sum_{i=1}^{\tilde{m}} \va_{i,1}^2-1\right\vert \geq t\right\}\leq 2 \exp\xkh{-\hat{c}\min{\left(\frac{t^2}{K^2},\frac{t}{K}\right)}\tilde{m}}
    \end{equation*}
    where $\hat{c}>0$ is an absolute constant and $K=\|\va_{i,1}^2\Vert_{\psi_1}={O}(1)$. 
 
    Taking $t \lesssim \sqrt{n/ \tilde{m}}$, one has 
    \begin{equation*}
    	\abs{J_{11}} = \left| \frac{1}{\tilde{m}}\sum_{i=1}^{\tilde{m}} \va_{i,1}^2-1\right\vert \abs{\vz_k^{||}} \lesssim \sqrt{\frac{n}{\tilde{m}}} \abs{\vz_k^{||}} \le \frac{c}{4 \log {\tilde{m}}} \abs{\vz_k^{||}}  
    \end{equation*}
    for a sufficient small constant $c>0$ with probability at least $1 - {O}(\exp\xkh{-c_1n})$ provided that   $\tilde{m}\gtrsim n\log^2{\tilde{m}}$. 
   For the second term $J_{12}$, Lemma \ref{Chen 2019, lamma 11} reveals that  
    \begin{equation*}
        \abs{J_{12}} = \left| \frac{1}{\tilde{m}}\sum_{i=1}^{\tilde{m}}\va_{i,\perp}^{\T}\vz_k^{\perp}\va_{i,1} \right\vert \lesssim \frac{1}{\tilde{m}}\left(\sqrt{V_1\log {\tilde{m}}}+B_1\log \tilde{m}\right)
    \end{equation*}
    with probability at least $1-{O}({\tilde{m}}^{-20})$  where $V_1$ and $B_1$ satisfy 
    \begin{equation*}
         V_1=\mathbb{E}_{\va_i\sim \mathcal{N}(\mathbf{0},\mathbf{I}_n)} \sum_{i=1}^{\tilde{m}}(\va_{i,\perp}^{\T}\vz_k^{\perp})^2\va_{i,1}^2 \quad \mbox{and} \quad B_1=\max_i{|\va_{i,\perp}^{\T}\vz_k^{\perp}\vert |\va_{i,1}}\vert.
    \end{equation*}
    The independence of $\va_i$ and $\vz_k$ implies $  V_1={\tilde{m}}\|\vz_k^{\perp}\Vert^2$. Then it follows from \eqref{eq:max_a_1} that $B_1\lesssim \log \tilde{m} \|\vz_k^{\perp}\Vert$ with high probability. When $ \tilde{m}\geq Cn\log^3{\tilde{m}} \log n$, one has 
    \begin{align*}
        \left| \frac{1}{\tilde{m}}\sum_{i=1}^{\tilde{m}}\va_{i,\perp}^{\T}\vz_k^{\perp}\va_{i,1} \right\vert 
        & \lesssim \sqrt{\frac{\log {\tilde{m}}}{\tilde{m}}}\|\vz_k^{\perp}\Vert+\frac{\log^2{\tilde{m}}}{\tilde{m}}\|\vz_k^{\perp}\Vert \overset{\mbox{(i)}}{\lesssim} \sqrt{\frac{\log \tilde{m}}{\tilde{m}}}\|\vz_k^{\perp}\Vert \\
        &\lesssim  \sqrt{\frac{\log \tilde{m}}{\tilde{m}}} \frac{1}{\sqrt{n\log n}}  \|\vz_k^{\perp}\Vert \sqrt{n\log n} \overset{\mbox{(ii)}}{\le} \frac{c}{4 \log \tilde{m}}|\vz_k^{||}\vert,
    \end{align*}
    where (i) follows from $\frac{\log^2{\tilde{m}}}{\tilde{m}}\lesssim\sqrt{\frac{\log \tilde{m}}{\tilde{m}}}$ due to $ \tilde{m}\gtrsim \log^3{\tilde{m}}$, and (ii) utilizes that 
    \[\sqrt{\frac{n\log \tilde{m}\log n}{\tilde{m}}}\lesssim \frac{c}{\log \tilde{m}}\] 
    and the induction hypotheses 
    \[\|\vz_k^{\perp}\Vert\leq 2, \quad |\vz_k^{||}\vert \geq \frac{1}{2\sqrt{n\log n}}.  \]
    In conclusion, it is established that
    \begin{equation*}
        |J_1\vert \le \abs{J_{11}} + \abs{J_{12}} \le \frac{c}{2 \log \tilde{m}}|\vz_k^{||}\vert 
    \end{equation*}
    with probability at least $1-{O}(e^{-c_1n})-{O}({\tilde{m}}^{-20})$ as long as $ \tilde{m}\geq Cn\log^3{\tilde{m}}\log n $ for a sufficiently large constant $C>0$. 
    
    \item For $J_2 $, one has  by Bernstein inequality (Lemma \ref{bernstein}),  
    \begin{align*}
        \mathbb{P}\left\{ \left|\frac{1}{\tilde{m}}\sum_{i=1}^{\tilde{m}} |\va_i^{\T}\vx\vert\sigma(\va_i^{\T}\vz_k)\va_{i,1}-\left(1-\frac{2\theta_k}{\pi}\right)-\frac{2\sin{\theta_k}}{\pi}\frac{\vz_k^{||}}{\|\vz_k\Vert} \right\vert \geq t\right\} \\
        \leq 2\exp\left[-\hat{c}\min(\frac{t^2}{K^2},\frac{t}{K})\tilde{m}\right]. 
    \end{align*}
    for any $t>0$ and some absolute constant $\hat{c}>0$ and \(K = \norms{|\va_1^{\T}\vx\vert\sigma(\va_1^{\T}\vz_k)\va_{1,1}}_{\psi_1} = O(1)\). 
    Take $t = \sqrt{\log \tilde{m}/\tilde{m}}$   to obtain  
    \begin{align*}
        |J_2\vert &= \left|\frac{1}{\tilde{m}}\sum_{i=1}^{\tilde{m}} |\va_i^{\T}\vx\vert\sigma(\va_i^{\T}\vz_k)\va_{i,1}-\left(1-\frac{2\theta_k}{\pi}\right)-\frac{2\sin{\theta_k}}{\pi}\frac{\vz_k^{||}}{\|\vz_k\Vert} \right\vert  \\
        &\le \sqrt{\frac{\log \tilde{m}}{\tilde{m}}} \le \frac{c}{2 \log \tilde{m}}|\vz_k^{||}\vert
    \end{align*}
    with probability at least $1-{O}(\tilde{m}^{-20})$ provided that \(\tilde{m} \ge C n \log^3 \tilde{m} \log n\). 
    Here, \(C>0\) is sufficiently large, and \(c>0\) is sufficiently small. 
    
    \item For $J_3 $,  divide it  as follows  
    \begin{align*}
    \begin{matrix}
        J_3=\underbrace{\left(\vz_k^{\perp}-\frac{1}{\tilde{m}}\sum_{i=1}^{\tilde{m}}\va_{i,\perp}^{\T}\vz_k^{\perp}\va_{i,\perp}\right)}_{J_{31}}-\underbrace{\frac{1}{\tilde{m}}\sum_{i=1}^{\tilde{m}}\va_{i,1}\vz_k^{||}\va_{i,\perp}}_{J_{32}}.
    \end{matrix}    
    \end{align*}
    We first consider  the term $J_{31}$.  Letting $\mathcal{N}$ be an $\epsilon$-net of the sphere $\mathcal{S}^{n-2}$,   
    Lemma \ref{bernstein} gives 
    \begin{equation*}
    	\mathbb{P} \xkh{\left|\frac{1}{\tilde{m}}\sum_{i=1}^{\tilde{m}}\va_{i,\perp}^{\T}\vz_k^{\perp}\va_{i,\perp}^{\T}\vy-{\vz_k^{\perp}}^{\T}\vy \right\vert \ge t} \le 2 \exp\xkh{-\hat{c} \min\xkh{\frac{t^2}{K^2}, \frac{t}{K}} \tilde{m}}
    \end{equation*}
    for fixed $\vz_k^{\perp}\in\mathbb{R}^{n-1}$ and any $\vy\in\mathcal{N}$.  Take \(t \lesssim \sqrt{n \log \tilde{m} / \tilde{m}}\) to get 
    \[\left|\frac{1}{\tilde{m}}\sum_{i=1}^{\tilde{m}}\va_{i,\perp}^{\T}\vz_k^{\perp}\va_{i,\perp}^{\T}\vy-{\vz_k^{\perp}}^{\T}\vy\right\vert \le \frac{3 c c_l}{16\log \tilde{m}} \le \frac{3 c}{16\log {\tilde{m}}}\norms{\vz_k^{\perp}}\]
    with probability exceeding \(1 - 2\exp(-\hat{c} n \log \tilde{m})\) as long as \(\tilde{m} \gtrsim n \log^3 \tilde{m}\). 
    Then apply  Lemma \ref{exercise 4.4.2}  to get  
    \begin{align*}
        \left\|\vz_k^{\perp}-\frac{1}{\tilde{m}}\sum_{i=1}^{\tilde{m}}\va_{i,\perp}^{\T}\vz_k^{\perp}\va_{i,\perp}\right\Vert &\leq \frac{1}{1-\epsilon}\sup_{\vy\in\mathcal{N}}\left|\frac{1}{\tilde{m}}\sum_{i=1}^{\tilde{m}}\va_{i,\perp}^{\T}\vz_k^{\perp}\va_{i,\perp}^{\T}\vy-{\vz_k^{\perp}}^{\T}\vy\right\vert \\
        & \le \frac{c}{4\log \tilde{m}}\norms{\vz_k^{\perp}}
    \end{align*}
    for \(\epsilon = 1/4\) with probability at least \(1 - 9^n \cdot 2\exp(-\hat{c} n \log \tilde{m}) \ge 1 - O(\tilde{m}^{-20})\). 
    In addition, through similar analysis, one can obtain that 
    \[\abs{J_{32}}  = \left\|\frac{1}{\tilde{m}}\sum_{i=1}^{\tilde{m}}\va_{i,1}\va_{i,\perp}\right \Vert|\vz_k^{||}\vert \le \frac{1}{1-\epsilon} \sup_{\vy\in\mathcal{N}}\left| \frac{1}{\tilde{m}}\sum_{i=1}^{\tilde{m}}\va_{i,1}\va_{i,\perp}^{\T}\vy\right\vert \cdot |\vz_k^{||}\vert \le \frac{c}{4\log \tilde{m}} \norms{\vz_k^{\perp}}\]
    with probability at least \(1 - O(m^{-20})\).  
    Therefore, it holds 
    \[\norms{J_3} \le \norms{J_{31}} + \norms{J_{32}} \le \frac{c}{2\log \tilde{m}} \norms{\vz_k^{\perp}} \]
    with the same probability. 
    
    \item For $J_4 $, similarly with $J_2$, if $\tilde{m}\geq Cn\log n\log^3{\tilde{m}}$, we obtain
    \begin{align*}
        \|J_4\Vert &=\left\|\frac{1}{\tilde{m}}\sum_{i=1}^{\tilde{m}} |\va_i^{\T}\vx\vert \sigma(\va_i^{\T}\vz_k)\va_{i,\perp}-\frac{2\sin{\theta_k}}{\pi}\frac{\vz_k^{\perp}}{\|\vz_k\Vert}\right\Vert  \\
                &\lesssim \frac{1}{\log \tilde{m}}\frac{1}{\sqrt{n\log n}} \le \frac{c}{2\log \tilde{m}} \|\vz_k^{\perp}\Vert
    \end{align*}
    with probability at least $1-{O}(e^{-c_4n})-{O}({\tilde{m}}^{-20})$.
\end{itemize}
Putting the previous bounds together yields \begin{align*}
	\abs{r_1(\vz_k)} &\le \abs{J_{1}} + \abs{J_{2}} \le \frac{c}{\log \tilde{m}} \abs{\vz_k^{||}}, \\
	\norms{r_\perp(\vz_k)} &\le \norms{J_{3}} + \norms{J_{4}} \le \frac{c}{\log \tilde{m}} \| \vz_k^{\perp}\Vert, 
\end{align*}
for a sufficiently small constant \(c>0\). Recalling \eqref{zk+1}, it holds \begin{subequations}\label{iterative of alpha beta}
	\begin{align}
		\alpha_{k+1}&=\left[1+\mu\left(-1+\frac{2}{\pi}\frac{\beta_k}{\alpha_k^2+\beta_k^2}+\zeta_k\right)\right]\alpha_k+\frac{2\mu}{\pi}\arcsin{\frac{\alpha_k}{\sqrt{\alpha_k^2+\beta_k^2}}}, \\
		\beta_{k+1}&=\left[1+\mu\left(-1+\frac{2}{\pi}\frac{\beta_k}{\alpha_k^2+\beta_k^2}+\rho_k\right)\right]\beta_k,  
	\end{align}
\end{subequations}
for some \(\abs{\zeta_k} \le c/\log \tilde{m}\) and \(\abs{\rho_k} \le c/\log \tilde{m}\) with probability at least \(1 - O(\exp(-\hat{c} n)) - O({\tilde{m}}^{-20})\). 
     
We next turn attention to the second conclusion \eqref{omega} of Lemma \ref{lemma:bound}. Notice that $\omega_k:=\arctan\xkh{\alpha_k/\beta_k}$. It then follows from (\ref{iterative of alpha beta})  that 
    \begin{subequations}
        \begin{align*}
            \alpha_{k+1} &= (1-\mu+\mu\zeta_k)\alpha_k+\frac{2}{\pi}\mu\frac{\tan{\omega_k}}{1+\tan^2{\omega_k}}+\frac{2}{\pi}\mu\omega_k, \\
            \beta_{k+1}  &= (1-\mu+\mu\rho_k)\beta_k+\frac{2}{\pi}\mu\frac{1}{1+\tan^2{\omega_k}}.
        \end{align*}
    \end{subequations}
    Noting that $\tan{\omega_{k+1}}=\frac{\alpha_{k+1}}{\beta_{k+1}}$, we derive
    \begin{align*}
        \tan{\omega_{k+1}}&=\frac{(1-\mu+\mu\rho_k)\alpha_k+\mu(\zeta_k-\rho_k)\alpha_k+\frac{2}{\pi}\mu\frac{\tan{\omega_k}}{1+\tan^2{\omega_k}}+\frac{2}{\pi}\mu\omega_k}{(1-\mu+\mu\rho_k)\beta_k+\frac{2}{\pi}\mu\frac{1}{1+\tan^2{\omega_k}}} \\
        &=\left(1+\frac{\mu(\zeta_k-\rho_k)r_k\sin{\omega_k}+\frac{2}{\pi}\mu {\omega_k}}{(1-\mu+\mu\rho_k)r_k\sin{\omega_k}+\frac{2}{\pi}\mu\sin{\omega_k}\cos{\omega_k}}\right)\tan{\omega_k} \\
        &\overset{\mbox{(i)}}{\geq} \left(1+\frac{\mu(\zeta_k-\rho_k)r_k\sin{\omega_k}+\frac{2}{\pi}\mu\sin{\omega_k}}{(1-\mu+\mu\rho_k)r_k\sin{\omega_k}+\frac{2}{\pi}\mu\sin{\omega_k}\cos{\omega_k}}\right)\tan{\omega_k}\\
        &\overset{\mbox{(ii)}}{\geq} \left(1+\frac{- 2 \mu (\abs{\zeta_k} + \abs{\rho_k})+\frac{2}{\pi}\mu}{2(1-\mu+\mu \abs{\rho_k})+\frac{2}{\pi}\mu}\right)\tan{\omega_k} \overset{\mbox{(iii)}}{\geq} \left(1+\frac{1}{4}\mu\right)\tan{\omega_k}.
    \end{align*}
    In inequality (i), we utilize $\omega_k\geq \sin{\omega_k}$ with \(\omega \ge 0\). Inequality (ii) comes from \(\cos \omega_k \le 1\) and the induction hypothesis (\ref{induction hypothesis}). In inequality (iii), we use the fact that 
    \[\frac{-(\abs{\zeta_k} + \abs{\rho_k})+\frac{1}{\pi}}{(1-\mu+\mu \abs{\rho_k})+\frac{1}{\pi}\mu}\geq\frac{1}{4}\]
    as long as \(\abs{\zeta_k}, \abs{\rho_k} \le c/\log \tilde{m} \ll 1/30\) for some sufficiently small constant \(c>0\).

\section{Proof of Lemma \ref{bound of r_k}}    
\label{proof for bound of r_k}
We first prove that $\norms{\vz_{k+1}}\le 2$.
For convenience, denote  
\[H(\alpha_k,\beta_k) := \frac{\alpha_k\beta_k}{\alpha_k^2+\beta_k^2}+\arcsin{\frac{\alpha_k}{\sqrt{\alpha_k^2+\beta_k^2}}} \stackrel{\mbox{(i)}}{=} \frac{1}{2}\sin{2\omega_k}+\omega_k \stackrel{\mbox{(ii)}}{\le} \frac{\pi}{2}\]
where (i) comes from the definition \(\omega_k = \arctan (\alpha_k/\beta_k)\), and (ii) holds since $\omega_k\in[0,\frac{\pi}{2}]$. 
Therefore, recalling (\ref{alpha beta approximate:a}), one has
    \begin{align*}
        \alpha_{k+1}&=(1-\mu+\mu\zeta_k)\alpha_k + \frac{2}{\pi}\mu H(\alpha_k,\beta_k) \leq (1-\mu+\mu \abs{\zeta_k})\alpha_k+\mu,
    \end{align*}
where \(\abs{\zeta_k}\le \frac{c}{\log \tilde{m}} \ll \frac{1}{30}\) for a sufficiently small constant \(c>0\) due to Lemma \ref{lemma:bound}. 
As long as \(\alpha_k  < \frac{1}{1 - \abs{\zeta_k}}\), it holds \begin{equation}\label{eq:alpha_k+1}
    \alpha_{k+1} < \frac{1}{1 - \abs{\zeta_k}} \le \frac{6}{5}.
\end{equation} 
For $\beta_k$, with the same approach, it follows from (\ref{alpha beta approximate:b}) that  
    \begin{align*}
        \beta_{k+1}=(1-\mu+\mu\rho_k)\beta_k+\frac{2}{\pi}\mu\frac{\beta_k^2}{\alpha_k^2+\beta_k^2}\leq (1-\mu+\mu\abs{\rho_k})\beta_k+\frac{2}{\pi}\mu.
    \end{align*}
where \(\abs{\rho_k}\le \frac{c}{\log \tilde{m}} \ll \frac{1}{30}\) for a sufficiently small constant \(c>0\) due to Lemma \ref{lemma:bound}. 
Then we consider three cases for \(\beta_k\). 
\begin{itemize}
	\item[-] If $\beta_k < \frac{2}{\pi(1-\abs{\rho_k})}$, then \(\beta_{k+1} \le \frac{2}{\pi(1-\abs{\rho_k})} \le 1\).  
	
	\item[-] If $\beta_k >  \abs{\rho_k}+\frac{2}{\pi}$, we have \begin{align*}
		\beta_{k+1} \le (1-\mu+\mu\abs{\rho_k})\beta_k+\frac{2}{\pi}\mu = \beta_k+\mu\left(-\beta_k + \abs{\rho_k} + \frac{2}{\pi}\right) \le \beta_k.
	\end{align*} 

	\item[-] If $\frac{2}{\pi(1-\abs{\rho_k})} \le \beta_k \le  \abs{\rho_k}+\frac{2}{\pi}$, one has \begin{align*}
		\beta_{k+1} \le& \beta_k+\mu\left(-\beta_k +  \abs{\rho_k} + \frac{2}{\pi}\right) = (1-\mu)\beta_k+\mu\left( \abs{\rho_k} + \frac{2}{\pi}\right) \\
        \le&  \abs{\rho_k} + \frac{2}{\pi}  \le 1.
	\end{align*} 
\end{itemize}
The fact \(\beta_0 \le \frac{7}{6}\) combined with the three cases reveals that \(\beta_{k+1} \le \frac{3}{2}\).  Therefore, it holds  
\[r_{k+1}^2 = \alpha_{k+1}^2+\beta_{k+1}^2\leq 4,  \]
namely \begin{equation}\label{result:upper}
	\|\vz_{k+1}\Vert \le 2.  
\end{equation}

Next, we prove $ \|\vz_{k+1}^{\perp}\Vert \ge c_{l}$. To this end, it is equivalent to prove $\beta_{k+1}\gtrsim 1$. When $k\leq T_{\gamma,2}$, by the definition of $T_{\gamma,2}$ one has $\beta_k\geq \frac{1}{2}\gamma$. When $T_{\gamma,2}<k\leq T_{\gamma}$, the iterate  (\ref{alpha beta approximate:b}) implies that 

\begin{equation*}
    \beta_{k+1}\geq (1-1.1\mu)\beta_k
\end{equation*}
as long as $ |\rho_k\vert\leq 0.1\mu$ and $\mu\le0.9
$. Then for any fixed $T_{\gamma,2}<k\leq T_{\gamma}$, it holds
\begin{equation}
    \beta_{k}\geq (1-1.1\mu)^{k-T_{\gamma,2}}\beta_{\gamma,2}\geq (1-1.1\mu)^{T_{\gamma}-T_{\gamma,2}}\beta_{\gamma,2}\geq \frac{\gamma}{2}(1-1.1\mu)^{T_{\gamma}-T_{\gamma,2}}.
\end{equation}
As long as $T_\gamma-T_{\gamma,2}\lesssim \frac{1}{\mu}$, one has $\beta_k\gtrsim 1$. Combined with (\ref{omega}), we obtain

\begin{equation}
    \left(1+\frac{1}{4}\mu\right)^{T_{\omega}-T_{\gamma,2}}\tan{\omega_{T_{\gamma,2}}}\leq\tan{\omega_{T_{\omega}}}. 
    \label{1/mu:1}
\end{equation}
Moreover, lemma \ref{bound of r_k} and the definition of $T_{\omega}$ give that 
\begin{equation}
    \frac{1}{2}\leq\frac{1}{\beta_{T_{\gamma,2}}} \quad \mbox{and} \quad \tan{\omega_{T_{\omega}}} < \tan{\left(\frac{\pi}{2}-\frac{1}{4}\gamma\right)}.
    \label{1/mu:2}
\end{equation}
Apply (\ref{1/mu:1}) and (\ref{1/mu:2}) to derive that
\begin{equation}
    \frac{1}{2}\left(1+\frac{1}{4}\mu\right)^{T_{\omega}-T_{\gamma,2}}\alpha_{T_{\gamma,2}}\leq\tan{\left(\frac{\pi}{2}-\frac{1}{4}\gamma\right)}.
    \label{1/mu:3}
\end{equation}
If $\alpha_{T_{\gamma,2}}\gtrsim 1$ holds, (\ref{1/mu:3}) reveals that $T_{\omega}-T_{\gamma,2}\lesssim \frac{1}{\mu} $. Therefore, combined with (\ref{eq:TminuTgam}) it follows 
\begin{equation*}
    T_{\gamma}-T_{\gamma,2}=(T_{\gamma}-T_{\omega})+(T_{\omega}-T_{\gamma,2})\lesssim \frac{1}{\mu}.
\end{equation*} 

    In the following, our aim is to proof $\alpha_{T_{\gamma,2}}\gtrsim1$. Let $\gamma > 0$ be some sufficiently small constant, and $\delta > 0$ be some  small constant. 

    In the following analysis, we partition Phase 1 (iterations up to $T_{\gamma}$) into several sub-stages. Figure 4 demonstrates these sub-stages through numerical experiments conducted with parameters $n = 1200$, $ m = 12n$, $\delta=0.2$, $\mu = 0.5$, and $\gamma = 0.1$. The plot reveals intricate evolution patterns of both $\alpha_k$ and $\beta_k$, which necessitated the identification of key temporal landmarks (marked by dashed vertical lines) to facilitate the proof.

\begin{figure}[htbp]
    \centering
    \includegraphics[width=0.5\linewidth]{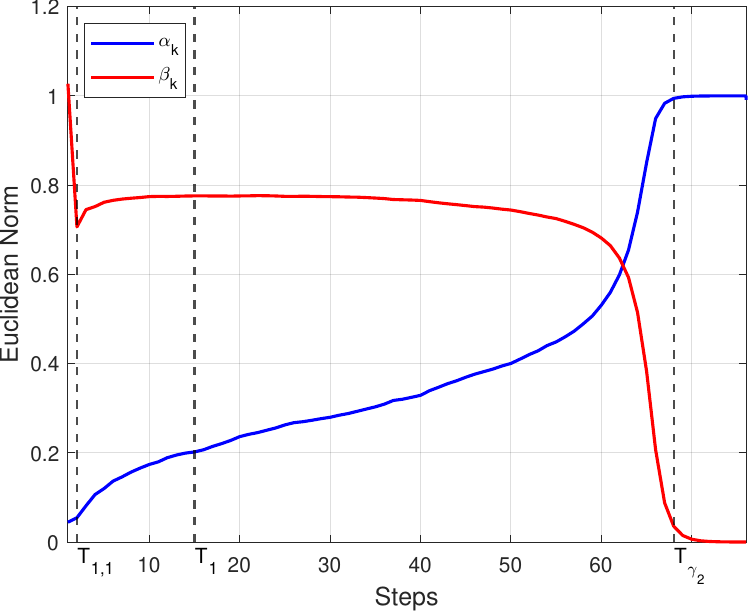}
    \caption{Illustration of the sub-stages for stage 1.}
    \label{fig:phases}
\end{figure} 

\begin{itemize}
    \item \textbf{Stage 1.1:} investigate the iterations $0\leq t\leq T_{1,1}$ with 
\begin{equation}
    T_{1,1}=\min\left\{t:\beta_{t+1}\leq\frac{3}{4}\right\}.
\end{equation}
\begin{fact}
    For $0\leq t\leq T_{1,1}$, it can be concluded that
\begin{subequations}\label{t:0-T_1,1}
\begin{align}
        \label{t:0-T_1,1:1}\beta_{T_{1,1}+1}&\leq\frac{3}{4}<\beta_{t}, \\
        \label{t:0-T_1,1:2}\beta_{t+1}&\leq \left(1-\frac{1}{16}\mu\right)\beta_t, \\
        \label{t:0-T_1,1:3}T_{1,1}&\lesssim\frac{1}{\mu}, \\
        \label{t:0-T_1,1:4}\alpha_t&\leq\delta, \\
        \label{t:0-T_1,1:5}\alpha_{t+1}&\geq\left(1+\frac{1}{20}\mu\right)\alpha_t, \\
        \label{t:0-T_1,1:6}\beta_{T_{1,1}+1}&\geq\frac{3}{4}\left(1-1.1\mu\right).  
    \end{align}
\end{subequations}
\end{fact}
A detailed proof of Fact 1 is provided in section \ref{proof:fact 1}. Fact 1 implies that in this substage, $\alpha_t$ keeps increasing and $\beta_t$ is decreasing. Additionally, this illustrates that $\beta_t$ drops to $\frac{3}{4}$ firstly and then $\alpha_t$ increase to $\delta$ secondly. In other words, $\alpha_t\leq\delta$ for $t=0,\ldots,T_{1,1}$. 
\item \textbf{Stage 1.2:} explore the iterations $T_{1,1}< t\leq T_1$ with 
\begin{equation}\label{def:T1}
    T_1=\min\{t:\alpha_{t+1}\geq \delta\}.
\end{equation}
\begin{fact}
    For $T_{1,1}< t\leq T_1$, it holds that
\begin{subequations}\label{t:T_1,1-_T_1}
\begin{align}
        \label{t:T_1,1-_T_1:1}\alpha_{t} &< \delta\leq\alpha_{T_1+1}, \\
        \label{t:T_1,1-_T_1:2}\frac{1}{3}\, &\le \beta_t \le \frac{3}{4}, \\  
        \label{t:T_1,1-_T_1:3}\alpha_{t+1} &\ge \left(1+\frac{1}{40}\mu\right)\alpha_t,\\
        \label{t:T_1,1-_T_1:4}T_1 &\lesssim\frac{\log n}{\mu}, \\
        \label{t:T_1,1-_T_1:5}\alpha_{T_1+1} &\leq(1 + 3\mu)\alpha_{T_1}.  
    \end{align}
\end{subequations}

\end{fact}
The detailed proof of Fact 2 is provided in Section \ref{proof:fact 2}. Importantly, Fact 2 demonstrates that during this phase $\beta_t$ maintains a lower bound  on \(O(1)\)  while $\alpha_t$ exhibits monotonic growth.
\item \textbf{Stage 1.3:} We define this substage as containing all iterations $t$ such that $ T_1< t \leq T_{\gamma,2}$ with 
\begin{equation}\label{def:t gamma 2}
    T_{\gamma,2}=\min\left\{t:\beta_{t+1}\leq\frac{1}{2}\gamma\right\}.
\end{equation}
\begin{fact}
When $ T_1< t \leq T_{\gamma,2}$, one has  
\begin{subequations}\label{t:T_1-T_gamma,2}
\begin{align}
        \label{t:T_1-T_gamma,2:1}T_{\gamma,2}-T_1&\lesssim \frac{1}{\mu}, \\
        \label{t:T_1-T_gamma,2:2}\beta_{t}&\gtrsim 1, \\
        \label{t:T_1-T_gamma,2:3}T_{\gamma,2}&\lesssim\frac{\log n}{\mu}, \\
        \label{t:T_1-T_gamma,2:4}\alpha_{T_{\gamma,2}}&\gtrsim 1 .   
\end{align}
\end{subequations}
\end{fact}

The proof of Fact 3 is presented in Section \ref{proof:fact 3}. Thus, we have accomplished our primary objective by demonstrating that at time $T_{\gamma,2}$, both $\alpha_{T_{\gamma,2}}$ and $\beta_{T_{\gamma,2}}$ possess constant lower bounds, as given in equations (\ref{t:T_1-T_gamma,2:2}) and (\ref{t:T_1-T_gamma,2:4}).

\end{itemize}

\subsection{Proof of Fact 1} \label{proof:fact 1}
The proof is structured as follows.  
\begin{itemize}
\item First of all, the definition \eqref{def:T1} gives that $\beta_{T_{1,1}+1}\leq\frac{3}{4}<\beta_{T_\gamma}. $
\item Additionally, the iterate  (\ref{alpha beta approximate:b})  implies that
    \begin{align*}
        \beta_{t+1} &= \left[1+\mu\left(-1+\frac{2}{\pi}\frac{\beta_t}{\alpha_t^2 + \beta_t^2} + \abs{\rho_t}\right)\right]\beta_t \\
        &\le \left[1+\mu\left(-1+\frac{2}{\pi}\frac{1}{\beta_t} + \abs{\rho_t}\right)\right]\beta_t \\
        &\le \xkh{1-\frac{1}{16}\mu}\beta_t
    \end{align*}
    where the last line follows from 
    \[-1+\frac{2}{\pi}\frac{1}{\beta_t} + \abs{\rho_t}\leq-\frac{1}{16}\]
    with the proviso that  $\abs{\rho_t} \ll \frac{1}{30}$ and $\beta_t \ge \frac{3}{4}$.  

\item  To prove (\ref{t:0-T_1,1:3}), from (\ref{t:0-T_1,1:2}) one knows that   
\begin{equation*}
    \frac{3}{4}<\beta_{T_{1,1}} \le \left(1-\frac{1}{16}\mu\right)^{T_{1,1}}\beta_0 \le \left(1-\frac{1}{16}\mu\right)^{T_{1,1}}.  
\end{equation*}
This further implies that 
\[T_{1,1} \lesssim \frac{\log \frac{4}{3}}{ - \log (1 - \frac{1}{16}\mu)} \lesssim \frac{1}{\mu} .\]

\item For (\ref{t:0-T_1,1:4}), the iterate (\ref{alpha beta approximate:a}) gives  
\begin{align*}
    \alpha_{t+1}&\overset{\mbox{(i)}}{\leq} \left[1+\mu\left(-1+\frac{2}{\pi}\frac{1}{\beta_t} + \abs{\zeta_t}\right)\right]\alpha_t  +\mu\frac{2}{\pi} \cdot \frac{\alpha_t}{\beta_t} \overset{\mbox{(ii)}}{\leq} \left(1+\frac{3}{4}\mu\right)\alpha_t
\end{align*}
where the inequality (i) utilizes the facts  
\[\arctan{x}\leq{x}, \quad \frac{\beta_t}{\alpha_t^2+\beta_t^2}\leq \frac{1}{\beta_t}, \quad\text{and} \quad \abs{\zeta_t} \ll \frac{1}{30},\]
and (ii) employs (\ref{t:0-T_1,1:1}).  
Combining this and (\ref{t:0-T_1,1:3}), we obtain 
\begin{equation*}
    \alpha_{T_{1,1}}\leq \left(1+\frac{3}{4}\mu\right)^{T_{1,1}}|\alpha_0\vert\leq \left(1+\frac{3}{4}\mu\right)^{{O}\left(\frac{1}{\mu}\right)}\frac{\log n}{\sqrt{n}} \ll \delta.
\end{equation*}

\item In addition, the induction hypothesis (\ref{alpha beta approximate:a}) indicates that
    \begin{align*}
        \alpha_{t+1} \overset{\mbox{(i)}}{\geq} & \left[1+\mu\left(-1+\frac{2}{\pi}\frac{\beta_t}{\alpha_t^2+\beta_t^2}+\zeta_t\right)\right]\alpha_t+\frac{2}{\pi}\mu\frac{\alpha_t}{\sqrt{\alpha_t^2+\beta_t^2}} \\
        = \, & \left[1+\mu\left(-1+\frac{2}{\pi}F(\alpha_t, \beta_t) - \abs{\zeta_t}\right)\right]\alpha_t \\
        \overset{\mbox{(ii)}}{\geq} & \left[1+\mu\left(-1+\frac{2}{\pi} \cdot \frac{7}{4} - \abs{\zeta_t}\right)\right]\alpha_t \overset{\mbox{(iii)}}{\geq} \left(1+\frac{1}{20}\mu\right)\alpha_t,
    \end{align*}
    where \[F(\alpha_t,\beta_t) := \frac{\beta_t+\sqrt{\alpha_t^2+\beta_t^2}}{\alpha_t^2+\beta_t^2}.  \]
    In inequality (i), we use $\arcsin{x}\geq x$ for \(x \ge 0\). 
    Inequality (ii) follows from   the fact   that 
    \[F(\alpha_t,\beta_t) \ge  F(\delta, \beta_t) \ge \frac{7}{4}\]
    for $\alpha_t \le \delta \le \beta_t \le 1.04$ and $\delta \le 0.3$.  
    And inequality (iii) arises from   $\abs{\zeta_t} \ll \frac{1}{30}$   due to Lemma \ref{lemma:bound}. (\ref{t:0-T_1,1:5}) also implies that $\alpha_t$ keeps increasing during this process.
\item Finally, combine  (\ref{alpha beta approximate:b}), $\alpha_{T_{1,1}}\leq\delta$ and $\beta_{T_{1,1}}>\frac{3}{4}$ to establish 
\begin{align*}
    \beta_{T_{1,1}+1}&\geq \left[1+\mu\left(-1+\frac{2}{\pi}\frac{\beta_{T_{1,1}}}{\alpha_{T_{1,1}}^2 + \beta_{T_{1,1}}^2} - \abs{\rho_{T_{1,1}}}\right)\right]\beta_{T_{1,1}} \geq \frac{3}{4}\left(1-1.1\mu\right)  
\end{align*}

due to \(\frac{\beta_{T_{1,1}}}{\alpha_{T_{1,1}}^2 + \beta_{T_{1,1}}^2} \ge 0\) and \(\abs{\rho_{T_{1,1}}} \ll \frac{1}{30}\). 
\end{itemize} 

\subsection{Proof of Fact 2} \label{proof:fact 2}
\begin{itemize}
    \item Firstly, we deduce from the definition of $T_1$ that $\alpha_{t}<\delta\leq\alpha_{T_1+1}$ for $T_{1,1}<t\leq T_{1}$. 
    \item Then we prove  (\ref{t:T_1,1-_T_1:2}). \eqref{t:0-T_1,1:6} and  \eqref{t:0-T_1,1:1}  reveal that 
    \[\frac{1}{3} \le \frac{3}{4} (1 - 1.1 \mu) \le \beta_{T_{1, 1} + 1} \le \frac{3}{4}\]
    as long as \(\mu \le \frac 12\). We now divide into two cases for \(\beta_{t}\) with \(T_{1, 1} < t \le T_{1}\).  
    \begin{itemize}
    	\item[-] If \(\beta_{t} \in [\frac 12, \frac 34]\),  then the iterate (\ref{alpha beta approximate:b}) gives that \begin{align*}
    		\beta_{t+1} &= \left[1 + \mu\left(-1 + \rho_t\right)\right]\beta_t + \frac{2\mu}{\pi} \cdot \frac{\beta_t^2}{\alpha_t^2+\beta_t^2} \\
    		& \le \left[1 + \mu\left(-1 + \abs{\rho_t}\right)\right]\beta_t + \frac{2\mu}{\pi}  \le \frac{3}{4}
    	\end{align*} 
    	where the first inequality follows from \(\frac{\beta_t^2}{\alpha_t^2+\beta_t^2} \le 1\), and the second inequality utilizes \(\mu \ge 0\) and \(\abs{\rho_t}\ll \frac{1}{30}\) by Lemma \ref{lemma:bound}.  On the other hand, one has \begin{align*}
    		\beta_{t+1} &= \left[1 + \mu\left(-1 + \rho_t\right)\right]\beta_t + \frac{2\mu}{\pi} \cdot \frac{\beta_t^2}{\alpha_t^2+\beta_t^2} \\
    		& \ge \left[1 + \mu\left(-1 - \abs{\rho_t}\right)\right]\beta_t + \frac{2\mu}{\pi} \cdot \frac{3}{5}  \ge \frac{1}{3}
    	\end{align*} 
    	due to the facts \(\mu \le 0.5\), \(\abs{\rho_t}\ll \frac{1}{30}\), and 
    	\[\frac{\beta_t^2}{\alpha_t^2+\beta_t^2} \ge \frac{0.5^2}{\delta^2+0.5^2} \ge \frac{3}{5}\]
    	with \(\beta_t \ge \frac{1}{2}\) and \(\alpha_t \le \delta < 0.4\).  
    	Therefore, we have shown that \(\beta_{t+1} \in [\frac 13, \frac 34]\).  
    	
    	\item[-] If \(\beta_{t} \in [\frac 13, \frac 12]\),  then the iterate (\ref{alpha beta approximate:b}) implies that   \begin{align*}
    		\beta_{t+1} &= \left[1+\mu\left(-1+\frac{2}{\pi}\frac{\beta_t}{\alpha_t^2 + \beta_t^2} - \rho_t \right)\right]\beta_t \\
    		&\le \left[1+\mu\left(-1+\frac{2}{\pi}\frac{1}{\beta_t} + \abs{\rho_t} \right)\right]\beta_t \\
    		&\le \left[1+\mu\left(-1+\frac{6}{\pi} + \frac{1}{30} \right)\right] \cdot \frac 12  \le  \frac{3}{4}  
    	\end{align*}
    	with the proviso that \(\frac 13 \le \beta_t \le \frac 12\) and \(\abs{\rho_t} \ll \frac{1}{30}\). 
    	Moreover,  it also follows from \eqref{alpha beta approximate:b} that   \begin{align*}
    		\beta_{t+1} =& \left[1+\mu\left(-1+\frac{2}{\pi}\frac{\beta_t}{\alpha_t^2 + \beta_t^2} - \rho_t \right)\right]\beta_t \\
    		\overset{\mbox{(i)}}{\ge}& \left[1+\mu\left(-1+\frac{2}{\pi}\cdot \frac{50}{29} - \frac{1}{30} \right)\right]\beta_t \ge  \left(1 + \frac{1}{16}\mu\right) \beta_t ,   
    	\end{align*}
    	where (i) holds since \(\abs{\rho_t} \ll \frac{1}{30}\) and  
    	\[\frac{\beta_t}{\alpha_t^2 + \beta_t^2} \ge \frac{0.5}{\delta^2 + 0.5^2} \ge \frac{50}{29}\]
    	with \(0 < \alpha_t \le \delta < 0.2\) and \(\frac 13 \le \beta_t \le \frac 12\).  
    	This justifies that \(\beta_{t+1} \in [\frac 13, \frac 34]\) in this case.  
    \end{itemize}
    Combining the preceding two cases establishes the claim \(\beta_{t+1} \in [\frac 13, \frac 34]\) for all \(T_{1, 1} < t \le T_{1}\).

\item Furthermore, the iterate (\ref{alpha beta approximate:b}) gives \begin{align}\label{beta t+1}
	\beta_{t+1} &\ge  \left[1+\mu\left(-1+\frac{2}{\pi}\frac{\beta_t}{\alpha_t^2 + \beta_t^2} - \abs{\rho_t}\right)\right]\beta_{t} \ge  \left(1-\frac{1}{5}\mu\right) \beta_{t} 
\end{align}
because of \(\abs{\rho_t} \ll \frac{1}{30}\) and 
\[\frac{\beta_t}{\alpha_t^2 + \beta_t^2} \ge \frac{0.75^2}{\delta^2 + 0.75^2} \ge \frac{55}{42}\]
with \(\delta < 0.1\) as well as \(\beta_t \in [\frac 13, \frac 34]\) from (\ref{t:T_1,1-_T_1:2}).  Additionally, it follows from \eqref{omega} that 
\[\frac{\alpha_{t+1}}{\beta_{t+1}} \ge \left(1+\frac{1}{4}\mu\right) \frac{\alpha_{t}}{\beta_{t}},  \]
which combined with \eqref{beta t+1} reveals that 
\[\alpha_{t+1} \ge \left(1+\frac{1}{40}\mu\right)\alpha_t\]
due to \(\mu \le 0.5\).

\item In light of (\ref{t:T_1,1-_T_1:3}) and \eqref{initial}, we have 
\[\xkh{1 + \frac{1}{25} \mu}^{T_{1}} \frac{1}{2\sqrt{n \log n}} \le \delta, \]
which implies that 
\[T_{1} \le \frac{\log \xkh{ 2\sqrt{n \log n} } - \log \delta}{\log\xkh{1 + \frac{1}{15}\mu}} \lesssim \frac{\log n}{\mu}.\]

\item Finally, apply (\ref{alpha beta approximate:a}), (\ref{t:T_1,1-_T_1:2}) and 
\[\arcsin{\frac{\alpha_t}{\sqrt{\alpha_t^2+\beta_t^2}}}\leq\frac{\alpha_t}{\beta_t}\]
to derive that 
\begin{align*}
    \alpha_{T_1+1}&\leq\left[1+\mu\left(-1+\frac{2}{\pi}\frac{1}{\beta_{T_1}} + \abs{\zeta_{T_1}}\right) \right]\alpha_{T_1}+\frac{2}{\pi}\mu\frac{\alpha_{T_1}}{\beta_{T_1}}\leq(1 + 3\mu)\alpha_{T_1}.
\end{align*}
due to \(\beta_{T_1} \ge \frac{1}{3}\) and \(\abs{\zeta_{T_1}} \ll \frac{1}{30}\).   
\end{itemize}

\subsection{Proof of Fact 3} \label{proof:fact 3}
\begin{itemize}
\item First of all, the conclusion (\ref{omega}) and definition (\ref{def:t gamma 2})  give that
\begin{equation}\label{omega:T_gamma,2-T_1}
    \frac{4}{\gamma} \ge \frac{\alpha_{T_{\gamma,2}}}{\beta_{T_{\gamma,2}}} 
    \ge \left(1+\frac{1}{4}\mu\right)^{T_{\gamma,2}-T_1} \frac{\alpha_{T_1}}{\beta_{T_1}} 
    \overset{\mbox{(i)}}{\ge} \left(1+\frac{1}{4}\mu\right)^{T_{\gamma,2}-T_1} \frac{4\delta}{9}
\end{equation}

where (i) utilizes \(\beta_{T_1} \le \frac 34\)  and 
\[\alpha_{T_1} \ge \frac{\alpha_{T_1+1}}{1 + 3\mu} \ge \frac{\delta}{3}\]
from \eqref{t:T_1,1-_T_1:2} and \eqref{t:T_1,1-_T_1:5} with \(\mu \le \frac{1}{2}\).  
Therefore, one has 
\begin{equation*}
    T_{\gamma,2}-T_1 \le \frac{\log \frac{9}{\delta\gamma}}{\log (1 + \frac 14 \mu)} \lesssim \frac{1}{\mu}, 
\end{equation*}
and, as a consequence, \(T_{\gamma,2} \lesssim \frac{\log n}{\mu}\).  

\item To prove (\ref{t:T_1-T_gamma,2:2}), it follows from  (\ref{alpha beta approximate:b}) and (\ref{t:T_1-T_gamma,2:1}) that  
\begin{align*}
    \beta_{t} \overset{\mbox{(i)}}{\ge}&  \left[1 + \mu\left(-1 + \frac{2}{\pi} \cdot \frac{\beta_k}{\alpha_k^2+\beta_k^2} - \abs{\rho_k}\right)\right]\beta_k \\
    \ge& \ (1-1.1\mu)^{t-T_1}\beta_{T_1}\geq(1-1.1\mu)^{T_{\gamma,2}-T_1}\beta_{T_1}\gtrsim 1,  
\end{align*}
where (i) comes from \(\abs{\rho_t} \ll \frac{1}{30}\) and \(\frac{\beta_t}{\alpha_t^2+\beta_t^2} \ge 0\).   

\item For (\ref{t:T_1-T_gamma,2:4}), apply \eqref{omega}, \eqref{result:upper}, \eqref{def:T1}, \eqref{def:t gamma 2}, \eqref{t:T_1-T_gamma,2:1} and \eqref{t:T_1-T_gamma,2:2} to derive that  
\[\alpha_{T_{\gamma,2}}
\ge \left(1+\frac{1}{4}\mu\right)^{T_{\gamma,2}-T_1-1} \frac{\alpha_{T_1+1}}{\beta_{T_1+1}} \beta_{T_{\gamma,2}}  \gtrsim 1.   \]
\end{itemize}

\bibliographystyle{elsarticle-num-names}
\label{sec:Appendix}

\end{document}